\documentclass[10pt,reqno]{amsart}

\usepackage{setspace}
\setdisplayskipstretch{1.5}
\usepackage[T1]{fontenc}
\usepackage{lmodern}

\usepackage{mathtools}
\usepackage{amsmath,amsfonts,amsbsy,amsgen,amscd,mathrsfs,amssymb,amsthm}
\usepackage{enumerate}
\usepackage{bm}
\usepackage{calc}

\usepackage[usenames,dvipsnames]{xcolor}
\usepackage[colorlinks=true,citecolor=blue,linkcolor=blue]{hyperref}

\newtheorem{thm}{Theorem}[section]
\newtheorem{lem}[thm]{Lemma}

\newtheorem{prop}[thm]{Proposition}
\newtheorem{cor}[thm]{Corollary}
\newtheorem*{triv}{Trivial observation}

\theoremstyle{definition}

\newtheorem{defn}[thm]{Definition}

\newtheorem*{rem*}{Remark}

\numberwithin{equation}{section} 
\numberwithin{figure}{section}
\numberwithin{table}{section}

\makeatletter

\newcommand{\bX}{\mathbb{X}}
\newcommand{\bG}{\mathbb{G}}

\begin{document}

\title{On the subgaussian comparison theorem}

\author[Van Handel]{Ramon van Handel}
\address{Department of Mathematics, Princeton University, Princeton, NJ 
08544, USA}
\email{rvan@math.princeton.edu}

\begin{abstract}
The aim of this expository note is to prove that any $1$-subgaussian 
random vector is dominated in the convex ordering by a universal constant 
times a standard Gaussian vector. This strengthens Talagrand's celebrated 
subgaussian comparison theorem. The proof combines a tensorization 
argument due to J.~Liu with ideas that date back to the work of 
Fernique.
\end{abstract}

\subjclass[2020]{60E15; 
                 60G15} 

\maketitle

\raggedbottom

\section{Introduction}
\label{sec:intro}

A random vector $X$ in $\mathbb{R}^n$ is said to be $1$-subgaussian if
$\mathbf{E}[X]=0$ and
$$
	\mathbf{P}\big[ |\langle v,X\rangle| > x \big]
	\le 2e^{-x^2/2}
$$
for all $x\ge 0$ and $v\in S^{n-1}$, that is, if it is centered and the 
tail probabilities of its linear projections are dominated by 
those of a standard Gaussian random variable. The main 
result of this note is that this weak form of domination implies a much 
stronger form of domination for the distribution of $X$.

\begin{thm}
\label{thm:main}
Let $X$ be any $1$-subgaussian random vector in $\mathbb{R}^n$ and 
$G\sim N(0,I_n)$ be a 
standard Gaussian vector in $\mathbb{R}^n$. Then
$$
	\mathbf{E}[f(X)] \le \mathbf{E}[f(cG)]
$$
for every convex function $f:\mathbb{R}^n\to\mathbb{R}$, where $c$ is
a universal constant.\footnote{%
As every convex function is lower bounded by an affine function, the 
expectations $\mathbf{E}[f(X)]$ and $\mathbf{E}[f(cG)]$ are well defined 
for every convex function $f$ and take values in $(-\infty,+\infty]$.}
\end{thm}

As we will recall below, the conclusion of Theorem \ref{thm:main} for 
\emph{$1$-homogeneous} convex functions is a direct consequence of the 
celebrated majorizing measure theorem of Talagrand \cite[\S 3]{Tal87}. 
That such a comparison principle holds for \emph{arbitrary} convex 
functions however appears to have been overlooked. This stronger form of 
domination is fundamentally more powerful and leads to a better 
structural understanding of subgaussian vectors. For example, the 
following corollary provides an equivalent formulation of
Theorem \ref{thm:main} by a 
classical result of Strassen \cite{Str65}.

\begin{cor}
There is a universal constant $c$ such that for every $1$-subgaussian 
vector $X$ in $\mathbb{R}^n$, we can construct $X$ and a standard 
Gaussian vector $G\sim N(0,I_n)$ on a common probability space such that 
$X = c\mathbf{E}[G|X]$.
\end{cor}

Theorem \ref{thm:main} will follow almost immediately by observing its 
connection with some old and recent ideas in the study of suprema of 
random processes. Beside the formulation of Theorem \ref{thm:main} and 
the more general Theorem \ref{thm:mainprocess} below, the 
expository aim of this note is to draw attention to these developments.

\subsection{Random processes}
\label{sec:introprocess}

We begin by formulating a more general form of the subgaussian comparison 
principle in terms of random processes.

To avoid irrelevant technicalities, we will consider only random 
processes defined on a finite index set $T$; the extension of the result
below to more general index sets is routine. 
Let $(G_t)_{t\in T}$ be any 
centered Gaussian process, and denote by
$$
	d(t,s) = \|G_t-G_s\|_2
$$
the associated natural metric on $T$. Let $(X_t)_{t\in T}$ be any 
centered random process that is subgaussian and dominated by $(G_t)_{t\in T}$
in the sense that
$$
	\mathbf{P}\big[|X_t-X_s|> x\big] \le 2e^{-x^2/2d(t,s)^2}
$$
for all $t,s\in T$. Finally, let $(m_t)_{t\in T}$ be any family of real 
numbers $m_t\in\mathbb{R}$ defined on the same index set. We will prove 
the following.

\begin{thm}
\label{thm:mainprocess}
For any centered Gaussian 
process $(G_t)_{t\in T}$, centered random process $(X_t)_{t\in T}$, and 
$(m_t)_{t\in T}$ satisfying the above assumptions, we have
$$
	\mathbf{E}
	\bigg[
	\sup_{t\in T}\big\{ X_t + m_t\big\}
	\bigg] \le
	\mathbf{E}
	\bigg[
	\sup_{t\in T}\big\{ cG_t + m_t\big\}
	\bigg],
$$
where $c$ is a universal constant.
\end{thm}

Theorem \ref{thm:main} follows readily from Theorem 
\ref{thm:mainprocess}. Indeed, applying Theorem \ref{thm:mainprocess} 
with $T\subset\mathbb{R}^n$, $X_t = \langle t,X\rangle$, $G_t=\langle 
t,G\rangle$ yields the conclusion of Theorem \ref{thm:main} for any 
function $f(x)=\sup_{t\in T}\{\langle t,x\rangle + m_t\}$ that is a 
finite maximum of affine functions. As any convex function 
$f:\mathbb{R}^n\to\mathbb{R}$ is the limit of an increasing sequence of 
functions of this form, Theorem \ref{thm:main} follows by monotone 
convergence.

In the same manner, the special case of Theorem \ref{thm:mainprocess} 
with $m_t\equiv 0$ yields the conclusion of Theorem \ref{thm:main} for 
functions of the form $f(x)=\sup_{t\in T}\langle t,x\rangle$, that is, 
for $1$-homogeneous convex functions. This special case follows 
from the celebrated majorizing measure theorem of Talagrand 
\cite{Tal87}, which states that
$$
	\frac{1}{c}
	\mathbf{E}
	\bigg[
	\sup_{t\in T} X_t
	\bigg]
	\le
	\gamma_2(T,d)
	\le
	c
	\mathbf{E}
	\bigg[
	\sup_{t\in T} G_t
	\bigg]
$$
where $\gamma_2(T,d)$ is an explicit functional that is defined in 
terms of the geometry of the metric space $(T,d)$. As 
$(G_t)_{t\in T}$ is itself a subgaussian process, replacing $(X_t)_{t\in 
T}$ by $(G_t)_{t\in T}$ on the left-hand side of this inequality shows 
that $\gamma_2(T,d)$ characterizes the expected supremum of 
any \emph{centered} Gaussian process up to a universal constant.
Since $\gamma_2(T,d)$ is difficult to compute in concrete situations,
however, the application of the majorizing measure theorem as a 
subgaussian comparison principle has proved to be one of its most useful 
features in practice.

Theorem \ref{thm:mainprocess} naturally leads us to seek a generalization 
of the majorizing measure theorem to \emph{non-centered} Gaussian 
processes. Despite that the suprema of non-centered processes arise in 
many applications, the problem of achieving sharp bounds for such 
processes does not appear to have been discussed in the literature. A 
common method for handling non-centered processes, the ``peeling 
device'', is to split $T$ into slices on which the value of $m_t$ is 
roughly constant and to estimate the supremum on each slice separately; 
see, e.g., \cite{vdG00}. While effective in various applications, such a 
procedure need not lead to sharp bounds.

We presently aim to explain that a form of the majorizing measure theorem 
for non-centered Gaussian processes is nonetheless already implicitly 
contained in another, largely forgotten, part of Talagrand's paper 
\cite[\S 4]{Tal87}.

\subsection{Fernique's functional}
\label{sec:triv}

We will need the following notion. Here and throughout this note,
$\|\cdot\|$ will denote the Euclidean norm.

\begin{defn}
Let $T$ be a finite set, and let $P_X$ and $\mu$ be probability measures
on $\mathbb{R}^T$ and $T$, respectively, with $\int \|x\|\, P_X(dx)<\infty$.
We define
$$
	\mathscr{F}(P_X,\mu) =
	\sup_{\substack{X\sim P_X\\ Z\sim\mu}} \mathbf{E}[X_Z],
$$
where the supremum is over all couplings of $P_X$ and $\mu$.
Given a random process $X=(X_t)_{t\in T}$, we will also write
$\mathscr{F}(X,\mu)=\mathscr{F}(P_X,\mu)$ where $P_X$ is the law of $X$.
\end{defn}

In other words, the quantity $\mathscr{F}(X,\mu)$ is the largest expected 
value of the random process $X$ evaluated at a random index with 
distribution $\mu$. This functional was first introduced by Fernique 
\cite{Fer76,Fer81} as a tool for understanding the expected suprema of 
random processes; some additional comments on the original motivation 
behind this quantity can be found in section \ref{sec:fernique} below.

Returning to the setting of Theorem \ref{thm:mainprocess}, we now make a

\begin{triv}
We can write
$$
	\mathbf{E}
	\bigg[
	\sup_{t\in T}\big\{ X_t + m_t\big\}
	\bigg] =
	\sup_\mu
	\bigg\{\mathscr{F}(X,\mu) +
	\int m_t\, \mu(dt)
	\bigg\},
$$
where the supremum is taken over all probability measures $\mu$ on $T$.
\end{triv}

\begin{proof}
The right-hand side can be equivalently expressed as
$$
	\sup_\mu \sup_{\substack{X\sim P_X\\ Z\sim\mu}} \mathbf{E}[X_Z
	+m_Z].
$$
This quantity is clearly upper bounded by $\mathbf{E}[\sup_{t\in T}
\{X_t+m_t\}]$. It is also lower bounded by it, as we can choose 
$(X,Z)$ so that $Z$ is a maximizer of $(X_t+m_t)_{t\in T}$.
\end{proof}

In the final section of his paper \cite[\S 4]{Tal87}, by an elaboration 
of the methods used to prove the majorizing measure theorem, Talagrand also 
provides a characterization of Fernique's functional 
$\mathscr{F}(G,\mu)$ for any centered Gaussian process $G$ and measure 
$\mu$, up to a universal constant, in terms of a certain geometric 
functional $\mathrm{I}_\mu(T,d)$ (for example, one may take the quantity 
$Q_3$ in \cite[Theorem 30]{Tal87} as its definition). When combined with 
the  above trivial observation, this provides the following majorizing 
measure theorem for non-centered Gaussian processes:
$$
        \mathbf{E}
        \bigg[
        \sup_{t\in T}\bigg\{ \frac{1}{c}G_t + m_t\bigg\}
        \bigg] \le
	\sup_\mu \bigg\{ \mathrm{I}_\mu(T,d) +
        \int m_t\, \mu(dt)
        \bigg\}
	\le
        \mathbf{E}
        \bigg[
        \sup_{t\in T}\big\{ cG_t + m_t\big\}
        \bigg].
$$
To complete the proof of Theorem \ref{thm:mainprocess}, it only remains 
to show that the first inequality remains valid if the Gaussian process
$(G_t)_{t\in T}$ is replaced by the subgaussian process $(X_t)_{t\in T}$
on the left-hand side. It seems likely that the methods of Fernique and 
Talagrand can be used to show that this is the case, but this is not 
immediately obvious from the proof that is presented in \cite{Tal87}.

Instead of pursuing this route, we aim to draw attention to a striking 
new approach to the majorizing measure theorem that was recently 
discovered by J.~Liu \cite{Liu25} which, as a byproduct, readily yields 
the comparison principle for $\mathscr{F}(X,\mu)$ (cf.\ \cite[Corollary 
2]{Liu25}) that is needed to complete the proof of Theorem 
\ref{thm:mainprocess}. The rest of this note is devoted to a short 
exposition of the proof of this result. A feature that is emphasized in 
our presentation is that it is now possible to prove comparison theorems 
such as Theorem \ref{thm:mainprocess} in an elementary manner that 
circumvents the need to achieve a complete geometric characterization of
the quantities in question.

\subsection{Organization of this note}

The remainder of this note is organized as follows. In section 
\ref{sec:fernique}, we recall Fernique's classical work on the suprema of 
Gaussian processes and include some historical comments. Section 
\ref{sec:liu} presents a simple tensorization principle that forms 
the basis of the work of J.~Liu. Finally, section \ref{sec:proof} 
combines these ingredients to complete the proof of Theorem 
\ref{thm:mainprocess}.

\section{On the work of Fernique}
\label{sec:fernique}

The systematic study of the suprema of general Gaussian processes dates 
back to the work of Dudley and Sudakov in the 1960s. The program 
of characterizing such suprema in geometric terms was 
subsequently taken up by Fernique. A major breakthrough, presented in 
Fernique's 1974 Saint Flour lectures \cite{Fer75}, was the complete 
solution of this problem for \emph{stationary} Gaussian processes.

Given a metric space $(T,d)$, the \emph{covering number} 
$N(T,d,\varepsilon)$ is the smallest number of $\varepsilon$-balls with 
respect to the metric $d$ that cover $T$. A random process 
$(G_t)_{t\in T}$ will be called \emph{stationary} if there is a group
$\Gamma$ that acts transitively on $T$ such that 
$(G_{\gamma(t)})_{t\in T}$ has the same distribution as $(G_t)_{t\in T}$ 
for every $\gamma\in\Gamma$.

\begin{thm}[Dudley; Fernique]
\label{thm:fernique}
Let the processes
 $(X_t)_{t\in T}$, $(G_t)_{t\in T}$ and the metric $d$
be as defined in 
section~\ref{sec:introprocess}, and suppose that $(G_t)_{t\in T}$ is stationary.
Then
$$
	\frac{1}{c}
	\mathbf{E}
	\bigg[
	\sup_{t\in T} X_t
	\bigg]
	\le
	\int_0^\infty \sqrt{\log N(T,d,\varepsilon)}\,d\varepsilon
	\le
	c
	\mathbf{E}
	\bigg[
	\sup_{t\in T} G_t
	\bigg].
$$
\end{thm}

\medskip

The first inequality is due to Dudley and the second is due to Fernique. 
The proofs of both inequalities are based on elementary chaining 
arguments that are essentially straightforward by modern standards. A 
simple direct proof of this theorem is sketched at the end of 
\cite[Chapter 6]{Led96}.

The stationarity assumption plays a key role in 
Theorem~\ref{thm:fernique}: it ensures that the geometry of $(T,d)$ is 
self-similar. Major difficulties arise when this assumption is dropped, 
since the process can then behave in a completely nonhomogeneous manner; 
indeed, it was known already in the 1960s that the suprema of 
non-stationary Gaussian processes cannot be characterized in terms of 
covering numbers \cite{Sud69}. To capture the nonhomogeneity, Fernique 
introduced a system of weights in his chaining arguments which led him to 
an improvement of the first inequality of Theorem~\ref{thm:fernique} in 
terms of the notion of a \emph{majorizing measure}. Fernique conjectured 
that this new upper bound is sharp for all centered Gaussian processes.

A conceptual obstacle to a proof of this conjecture was the lack of a 
clear probabilistic interpretation of the majorizing measure, which 
arises in a purely non-probabilistic manner in the upper bound. 
Fernique's intuition was that the majorizing measure should be closely 
connected to the distribution of the maximizer of the Gaussian process. 
The functional $\mathscr{F}(X,\mu)$ was introduced in \cite{Fer76,Fer81} 
in order to elucidate the relation between these notions; see \cite[\S 
3.3]{Fer81} and \cite[p.\ 105]{Tal87}. However, this approach does not 
appear to have led to significant progress. Not the least remarkable 
aspect of Talagrand's celebrated resolution of Fernique's conjecture 
\cite{Tal87} is that his proof was entirely geometric in nature, 
avoiding the need to understand the majorizing measure itself; indeed, 
majorizing measures play only an incidental role in the definitive 
contemporary treatment of this subject \cite{Tal21}.

Very recently, however, a remarkable idea of J.\ Liu \cite{Liu25} has led 
to an unexpected new proof of the majorizing measure theorem that 
completely bypasses the methods introduced in Talagrand's work. At the 
core of Liu's approach is the observation that, by a simple tensorization 
argument that is explained in the following section, the computation of 
$\mathscr{F}(X,\mu)$ can be reduced to computing the expected supremum of 
an auxiliary \emph{stationary} random process. The analysis of 
$\mathscr{F}(X,\mu)$ therefore reduces to the much more elementary 
setting of Theorem \ref{thm:fernique}. This provides a new probabilistic 
approach to majorizing measures that appears to be much closer in spirit 
to the program that was originally envisioned by Fernique.

\section{Liu's tensorization principle}
\label{sec:liu}

\subsection{Statement of the principle}

Let $T$ be a finite set, and let
$$
	\textstyle
	\mathcal{P}_K = \left\{\frac{1}{K}\sum_{i=1}^K \delta_{t_i}:
	t_1,\ldots,t_K\in T\right\}
$$
be the set of probability measures on $T$ whose atom probabilities 
are integer multiples of $\frac{1}{K}$.
Given any $K,N\in\mathbb{N}$ and $\mu\in\mathcal{P}_K$,
we let
$$
	\textstyle
	\mathcal{T}_N(\mu) =
	\left\{
	\mathbf{t}\in T^{NK} :
	\frac{1}{NK}\sum_{i=1}^{NK} \delta_{\mathbf{t}_i} = \mu
	\right\}
$$
be the set of sequences in which each $t\in T$ appears exactly 
$NK\mu(\{t\})$ times.

\begin{prop}[Liu's tensorization principle]
\label{prop:liu}
Let $T$ be a finite set, $X=(X_t)_{t\in T}$ be a
random process with $\max_t\|X_t\|_1<\infty$, and 
$\mu\in\mathcal{P}_K$. Define
$$
	\textstyle
	\bX_{\mathbf{t}} = \frac{1}{M}\sum_{i=1}^M
	X_{\mathbf{t}_i}^{(i)}
$$
for every $M\in\mathbb{N}$ and $\mathbf{t}\in T^M$,
where $X^{(1)},X^{(2)},\ldots$ are i.i.d.\ copies of $X$.
Then
$$
	\mathscr{F}(X,\mu) = 
	\lim_{N\to\infty}
	\mathbf{E}\bigg[
	\sup_{\mathbf{t}\in\mathcal{T}_N(\mu)}
	\bX_{\mathbf{t}}
	\bigg].
$$
\end{prop}

\medskip

The point here is that the random process 
$(\bX_{\mathbf{t}})_{\mathbf{t}\in\mathcal{T}_N(\mu)}$ is stationary. 
Indeed, let the symmetric group $\mathrm{S}_{NK}$ act on 
$\mathcal{T}_N(\mu)$ by defining 
$\sigma(\mathbf{t})=(\mathbf{t}_{\sigma(1)},\ldots,\mathbf{t}_{\sigma(NK)})$ 
for every $\mathbf{t}\in \mathcal{T}_N(\mu)$ and $\sigma\in 
\mathrm{S}_{NK}$. This action is clearly transitive. Moreover, as
$$
	\textstyle
	\bX_{\sigma(\mathbf{t})} 
	= \frac{1}{NK}\sum_{i=1}^{NK}
        X_{\mathbf{t}_i}^{(\sigma^{-1}(i))}
$$
and $X^{(1)},X^{(2)},\ldots$ are exchangeable, the processes 
$(\bX_{\sigma(\mathbf{t})})_{\mathbf{t}\in\mathcal{T}_N(\mu)}$ and 
$(\bX_{\mathbf{t}})_{\mathbf{t}\in\mathcal{T}_N(\mu)}$ have the same 
distribution for every $\sigma\in \mathrm{S}_{NK}$. Thus 
Proposition~\ref{prop:liu} reduces the computation of Fernique's 
functional for an arbitrary random process to the computation of the 
expected supremum of a stationary process.

Proposition \ref{prop:liu} is a variant of \cite[Lemma 5]{Liu25}. For 
completeness, we include a short proof of this result in the remainder of 
this section. We emphasize that this requires no new idea as compared to 
the arguments in \cite{Liu25}.

\subsection{Two simple lemmas}

Recall that the 
Wasserstein distance between probability measures $P_X,P_X'$ on 
$\mathbb{R}^T$ is defined as
$$
	W_1(P_X,P_X') = \inf_{\substack{X\sim P_X \\ X'\sim P_X'}}
	\mathbf{E}\|X-X'\|,
$$
where the infimum is taken over all couplings of $P_X$ and $P_X'$.
The following straightforward continuity property will be used below.

\begin{lem}
\label{lem:cont}
Let $T$ be a finite set, $\mu$ be a probability measure on $T$, and
$P_X,P_X'$ be probability measures on 
$\mathbb{R}^T$ with $\int \|x\|\,P_X(dx)<\infty$, $\int 
\|x\|\,P_X'(dx)<\infty$. Then
$$
	|\mathscr{F}(P_X,\mu)-\mathscr{F}(P_X',\mu)|
	\le W_1(P_X,P_X').
$$
\end{lem}

\begin{proof}
Given a pair of random processes $(X,X')\sim P_{XX'}$ distributed according to any 
coupling of $X\sim P_X$ and $X'\sim P_X'$, we can readily estimate
\begin{align*}
	\mathscr{F}(P_X,\mu)-\mathscr{F}(P_X',\mu)
	&=
	\sup_{\substack{(X,X')\sim P_{XX'} \\ Z\sim\mu}} \mathbf{E}[X_Z]
	-
	\sup_{\substack{(X,X')\sim P_{XX'} \\ Z\sim\mu}} \mathbf{E}[X_Z']
\\	& \le
	\sup_{\substack{(X,X')\sim P_{XX'} \\ Z\sim\mu}}
	\mathbf{E}[X_Z-X_Z'] \le
	\mathbf{E}\|X-X'\|,
\end{align*}
where the supremum is over all couplings of $P_{XX'}$ and $\mu$. By 
exchanging the role of $P_X$ and $P_X'$, the inequality 
remains valid if the take the absolute value of the left-hand side.
It remains to take the infimum over 
all couplings $(X,X')$.
\end{proof}

We also recall the following routine consequence of the law of large 
numbers and the metric properties of the Wasserstein distance.

\begin{lem}
\label{lem:lln}
Let $T$ be a finite set, and let $P_X$ be a probability measure on
$\mathbb{R}^T$ with $\int \|x\|\,P_X(dx)<\infty$.
Let $X^{(1)},X^{(2)},\ldots$ be i.i.d.\ copies of $X\sim P_X$.
Then
$$
	\lim_{N\to\infty}\mathbf{E}\big[W_1(P_X,\hat P^N)\big] = 0,
$$
where $\hat P^N=\frac{1}{N}\sum_{i=1}^N \delta_{X^{(i)}}$ 
denotes the empirical distribution.
\end{lem}

\begin{proof}
The law of large numbers implies both that $\hat P^N$ converges weakly to 
$P_X$
a.s., and that $\int \|x\|\,\hat P^N(dx)\to \int \|x\|\,P_X(dx)$ a.s.
These two properties together yield $W_1(P_X,\hat P^N)\to 0$ a.s.\ by 
\cite[Theorem 7.12]{Vil03}. 

To prove that the convergence also holds in expectation 
$\mathbf{E}[W_1(P_X,\hat P^N)]\to 0$, it remains to show that the sequence 
$(W_1(P_X,\hat P^N))_{N\ge 1}$ is uniformly integrable. This follows by 
noting that $W_1(P_X,\hat P^N) \le \int \|x\|\,P_X(dx) + \int \|x\|\,\hat 
P^N(dx)$, where the right-hand side converges in $L^1$ by the law of large 
numbers.
\end{proof}

\subsection{Proof of the tensorization principle}

\begin{proof}[Proof of Proposition \ref{prop:liu}] Let $P_X$ be the law of 
$X$, and define $\hat P^N$ as in Lemma \ref{lem:lln}. Fix an arbitrary 
$\mathbf{t}'\in\mathcal{T}_N(\mu)$. Then any coupling of $\hat P^{NK}$ and 
$\mu$ can be realized by sampling $(X^{(i)},\mathbf{t}_j')$ such that each 
pair $(i,j)$ is selected with probability $\frac{1}{NK}\Pi_{ij}$, where 
$\Pi$ is an $NK\times NK$ bistochastic matrix. Thus we have
$$
	\mathscr{F}(\hat P^{NK},\mu) =
	\sup_{\Pi\in \mathrm{B}_{NK}}
{\textstyle
	\frac{1}{NK}\sum_{i,j=1}^{NK}
	\Pi_{ij} X^{(i)}_{\mathbf{t}_j'} 
}
	=
	\sup_{\sigma\in\mathrm{S}_{NK}}
{\textstyle
	\frac{1}{NK}\sum_{i=1}^{NK}
	X^{(i)}_{\sigma(\mathbf{t}')_i}
}
	=
	\sup_{\mathbf{t}\in\mathcal{T}_N(\mu)} \bX_{\mathbf{t}},
$$
where we used that the set $\mathrm{B}_{NK}$ of bistochastic matrices is 
the convex hull of the set of permutation matrices by Birkhoff's theorem 
\cite[p.\ 5]{Vil03}. Taking the expectation and applying
Lemmas \ref{lem:cont} and \ref{lem:lln} concludes the proof.
\end{proof}

\section{Proof of Theorem \ref{thm:mainprocess}}
\label{sec:proof}

With the above ingredients in hand, the proof of Theorem 
\ref{thm:mainprocess} only requires some minor technicalities.
We need the following result that is similar to Lemma \ref{lem:cont}.

\begin{lem}
\label{lem:tv}
Let $T$ be a finite set, and let $P_X$ be a probability measure on 
$\mathbb{R}^T$ with $\int \|x\|\,P_X(dx)<\infty$. Then
$\mu\mapsto\mathscr{F}(P_X,\mu)$ is continuous in total
variation.
\end{lem}

\begin{proof}
Let $\mu,\mu'$ be probability 
measures on $T$ and let $(Z,Z')$ be any coupling of $Z\sim \mu$ and 
$Z'\sim\mu'$. Arguing as in the proof of Lemma \ref{lem:cont}, 
we obtain
$$
	|\mathscr{F}(P_X,\mu)-\mathscr{F}(P_X,\mu')|
	\le
	\sup_{X\sim P_X}
	\mathbf{E}[X_Z-X_{Z'}] \le
	2\sup_{X\sim P_X}
	\mathbf{E}[1_{Z\ne Z'} \|X\|]
$$
where the supremum is over all couplings of $P_X$ with the law of 
$(Z,Z')$. Estimating 
$$
	\mathbf{E}[1_{Z\ne Z'} \|X\|] \le
	r\mathbf{P}[Z\ne Z'] + \mathbf{E}[\|X\|1_{\|X\|>r}]
$$
and taking
the infimum over all couplings $(Z,Z')$, we can estimate for
any $r\ge 0$
$$
	|\mathscr{F}(P_X,\mu)-\mathscr{F}(P_X,\mu')| \le
	r \|\mu-\mu'\|_{\rm TV} + 2\,\mathbf{E}[\|X\|1_{\|X\|>r}]
$$
using the coupling characterization of
the total variation metric \cite[p.\ 7]{Vil03}.
\end{proof}

We can now conclude the proof.

\begin{proof}[Proof of Theorem \ref{thm:mainprocess}]
By the trivial observation in section \ref{sec:triv}, it suffices to 
prove
$$
	\mathscr{F}(X,\mu) \le c \mathscr{F}(G,\mu)
$$
for every probability measure $\mu$ on $T$.

Let us first fix $K,N\in\mathbb{N}$ and $\mu\in\mathcal{P}_K$, and 
consider 
the two random processes 
$(\bX_\mathbf{t})_{\mathbf{t}\in\mathcal{T}_N(\mu)}$ 
and $(\bG_\mathbf{t})_{\mathbf{t}\in\mathcal{T}_N(\mu)}$ as defined by
Proposition \ref{prop:liu}. Then
$$
\textstyle
	d_N(\mathbf{t},\mathbf{s}) := \|\bG_\mathbf{t}-\bG_\mathbf{s}\|_2 =
	\frac{1}{NK}\sqrt{
	\sum_{i=1}^{NK}
	d(\mathbf{t}_i,\mathbf{s}_i)^2},
$$
and it is elementary (see, e.g., \cite[Theorem 2.7.3]{Ver26}) that
$$
	\mathbf{P}\big[|\bX_\mathbf{t}-\bX_\mathbf{s}|> Cx\big] \le 
	2e^{-x^2/2d_N(\mathbf{t},\mathbf{s})^2}
$$
for a universal constant $C$. Since
$(\bG_\mathbf{t})_{\mathbf{t}\in\mathcal{T}_N(\mu)}$ is stationary,
we obtain
$$
	\mathscr{F}(X,\mu) =
	\lim_{N\to\infty}
	\mathbf{E}\bigg[
	\sup_{\mathbf{t}\in\mathcal{T}_N(\mu)}
	\bX_{\mathbf{t}}
	\bigg] \lesssim
	\lim_{N\to\infty}
	\mathbf{E}\bigg[
	\sup_{\mathbf{t}\in\mathcal{T}_N(\mu)}
	\bG_{\mathbf{t}}
	\bigg] =
	\mathscr{F}(G,\mu)
$$
by Theorem \ref{thm:fernique} and Proposition \ref{prop:liu}.
This proves the desired inequality for every $K\in\mathbb{N}$ and
$\mu\in\mathcal{P}_K$. It remains to note that the conclusion extends to 
an arbitrary probability measure $\mu$ on $T$ by continuity using Lemma 
\ref{lem:tv}.
\end{proof}

For the purpose of proving a subgaussian comparison theorem, the approach 
that we have followed here completely avoids the need to obtain a 
geometric characterization of $\mathscr{F}(G,\mu)$. The latter can also be 
achieved, however: using Proposition~\ref{prop:liu} and 
Theorem~\ref{thm:fernique}, this problem reduces to understanding the 
asymptotics of the covering numbers 
$N(\mathcal{T}_N(\mu),d_N,\varepsilon)$ as $N\to\infty$, which is a 
classical problem of coding theory \cite{CK11}. Such an analysis is 
developed in detail in the work of J.\ Liu \cite{Liu25}, leading to a new 
formulation and proof of the majorizing measure theorem.

\subsection*{Acknowledgments}

The author thanks Antoine Song for an interesting conversation which 
motivated the observations that are recorded in this note, Jingbo Liu 
for many discussions on his approach to the majorizing measure theorem,
and the referee for various helpful suggestions that improved the
presentation.
This work was supported in part by NSF grant DMS-2347954.

\bibliographystyle{abbrv}
\bibliography{ref}

\end{document}